\documentclass[12pt]{article}
\usepackage{latexsym,amssymb,amsmath,amsfonts,latexsym,amscd,amsthm}
\begin{document}

\newtheorem{ob}{Observation}[section]
    \newtheorem{lem}{Lemma}[section]
    \newtheorem{rem}{Remark}[section]
	\newtheorem{prop}{Proposition}[section]
\numberwithin{equation}{section}
\title{
\Large
\bf
Riemannian geodesics
    \newtheorem{cor}{Corollary}[section]
    \newtheorem{theo}{Theorem}[section]
    \newtheorem{defi}{Definition}[section]
    \newtheorem{exam}{Example}[section]
of semi Riemannian warped  product metrics} \author{
\small
Oriella M. Amici
and
\  Biagio C. Casciaro\\
\small  Dipartimento di Matematica
Universit{\`a} di Bari\\
\small Campus Universitario\\
\small Via Orabona 4, 70125 Bari, Italy
\\
\small casciaro@dm.uniba.it \& amici@dm.uniba.it
}
\maketitle

\begin{abstract}
 Let $(M_1,g_1)$ and $(M_2,g_2)$ be two
$C^\infty$--differentiable connected, complete Riemannian manifolds, 
$k:M_1\to\mathbb R$ a $C^\infty$--differentiable function, having $0<k_0<k(x)\leq K_0$, for any $x\in M_1$ and $g:=g_1-kg_2$ the
semi Riemannian metric on the product manifold $M:=M_1\times M_2$.
\par
We associate to $g$ a suitable family of Riemannian metrics $G_r+g_2$, with $r>-K_0^{-1}$, on $M$ and we call \emph{Riemannian} geodesics of $g$ the geodesics of $g$ which are geodesics of a metric of the previous family, via a suitable reparametrization. 
\par
Among the properties of these geodesics, we quote:
\par
For any $z_0=(x_0,y_0)\in M$ and for any $y_1\in M_2$ there exists a subset $A$ of $M_1$, such that all the geodesics of $g$ joining $z_0$ with a point $(x_1,y_1)$, with $x_1\in A$, are Riemannian.  
The Riemannian geodesics of $g$ determine a "partial" property of geodesic connection on $M$. 
Finally, we determine two new classes of semi Riemannian metrics (one of which includes some FLRM-metrics), geodesically connected by Riemannian geodesics of $g$.
\end{abstract}
\newpage
\section{Introduction}
Let $(M_1,g_1)$ and $(M_2,g_2)$ be two connected, complete, Riemannian manifolds.
\par
For the greater part of the paper, we shall use the assumption of the completeness of the two manifolds only to avoid to write a long and trivial series of inequalities.
\par 
Let $k:M_1\to{\mathbb R}$ be a $C^\infty$--differentiable function, bounded from below away from zero. 
\par
We consider the semi Riemannian warped product metric $g:g_1-kg_2$ and the family of Riemannian metrics $G_r+g_2$ on the manifold $M:=M_1\times M_2$, where $G_r:=(k^{-1}+r)g_1$ and $r>-K_0^{-1}:=k_1$, being $K_0:=\sup_{x\in M_1}\{k(x)\}$, if $k$ is bounded from above and $r>0:=k_1$ in the other case. 
\par
Then we prove that $M$ is complete with respect to the metric $G_r+g_2$  and that the geodesics of $G_r+g_2$, belonging to a suitable subset, determine geodesics of $g$, via a suitable reparametrization, for any $r>k_1$. 
\par
We call them \emph{Riemannian geodesics of $g$}. 
\par
We prove some properties of these geodesics and here we quote some of them as examples.
\par
Let us consider $z_0=(x_0,y_0)\in M$ and a geodesic $\zeta=(\gamma,\tau):[0,1]\to M$ of $g$, with $\gamma(0)=x_0$, $\tau(0)=y_0$, $\dot\gamma(0)=\widetilde X$ and $\dot\tau(0)=\widetilde Y$. If $k$ is bounded and
\[
g_1(\widetilde{X},\widetilde{X})>k(x_0)g_2(\widetilde{Y},\widetilde{Y})\frac{K_0-k(x_0)}{K_0}\ ;
\]
then $\zeta$ is a Riemannian geodesic of $g$.
\par 
An analogous statement holds, if $k$ is unbounded from above.

\par
A surprising property, being the Morse theory of Riemannian and semi Riemannian metrics quite different, is the following.
\par
Since $M_1$ and $M_2$ are connected and complete with respect to the respective Riemannian metrics $g_1$ and $g_2$,
the manifold $M_1$ is \emph{positive and negative geodesically connected with respect to $g$}; i.e., for any real number $r>k_1$, for any $z_0=(x_0,y_0)\in M$, for any $x_1\in M_1$ and for any geodesic $\nu:{\mathbb R}\to M_2$ of $g_2$, having $\nu(0)=y_0$, there exists $t_0\in{\mathbb R}$ such that the point $z_0$ and the point $(x_1,\nu(t_0))$ (and the point $(x_1,\nu(-t_0))$) can be joined by a Riemannian geodesic of $g$, obtained by reparametrizing a suitable geodesic of $G_r+g_2$.
\par
Analogously, the manifold $M_2$ is \emph{positive and negative geodesically connected with respect to $g$}, too.
\par
Hence, we shall say that \emph{$M$ is partially Riemannian connected with respect to $g$}
\par
More surprising are the following two results.
\par
\emph{If $M_1$ and $M_2$ are connected and complete with respect to the respective Riemannian metrics $g_1$ and $g_2$, if the dimension of $M_1$ is greater than one and $M_1$ is simply connected, if $g_1$ has a negative sectional curvature, if $k$ is bounded from below away from zero and if the Hessian of $k$ verifies a\\ suitable inequality $($see \eqref{ne}, below$)$, then $M$ is geodesically connected by means Riemannian geodesic of $g$.}
\par
\emph{If $M_1={\mathbb R}$, then $g$ is an FLRW--metric $($with speed of light $c=1)$ and $M$ is geodesically connected by Riemannian geodesic of $g$, provided $M_2$ connected and complete with respect to $g_2$ and $k$ bounded from below away from zero.}
\par
The FLRW--metrics are used in cosmology to study the early universe (see, e. g., \cite{BG}).
\par
The paper 
ends with an Appendix in which we determine a sufficient condition such that $G_r$ has negative sectional curvature, for any $r\in(k_1,+\infty)$.
\par
We conclude by noticing that the Levi--Civita connection of $g$ is not used in this paper, because it hides all the relations between the metric tensor $g$ and the Riemannian metric $G_r+g_2$.
\par
In this case, the Levi--Civita connection of $g_1+g_2$ allows us to use these relations.
\par
Hence, we consider this paper as a first application of the results obtained in \cite{ACF}, \cite{CF} and \cite{CK}.
\section{Preliminaries}
This Section contains the main geometric objects, which are needed in  the following.
\par
We also state some straightforward results.
\par
Let
$(M_1,g_1)$, $(M_2,g_2)$  be
two connected, complete, Riemannian manifolds and 
$\overset{\scriptscriptstyle1}\nabla$,
$\overset{\scriptscriptstyle2}\nabla$
the Le\-vi-Ci\-vi\-ta connections determined by the metrics
$g_1$ and $g_2$,
respectively.
\par
Let $k:M_1\to\mathbb R$ be a smooth bounded map.
\par
We suppose
\begin{eqnarray}
\label{boun} 
0<k_0:=\inf_{x\in M_1}\{k(x)\}\ .
\end{eqnarray}
On the manifold $M:=M_1\times M_2$, we consider the tensor $g:=g_1-k\cdot g_2$, which defines a \emph{semi Riemannian warped product metric}, having the signature equal to the dimension of $M_1$.
\par 
The geometry of warped product metrics is described in
details in~\cite{ON}.
\par
We shall set 
\[G_r:=(\frac{1}{k}+r)\cdot g_1
\]
and $G_r$ is a Riemannian metric on $M_1$, for any $r>k_1$, being $k_1:=-K_0^{-1}$ if $k$ is bounded and $K_0:=\sup_{x\in M_1}\{k(x)\}$, and $k_1:=0$ in the other case.
\par
Finally, we set  $I:=[0,1]$.
\par
From \cite{CK}, it follows. 
\begin{lem}
\label{th1}
A differentiable curve $\zeta=(\gamma,\tau):I\to M$ is a geodesic of $g$, if and only if it satisfies
the following system of ordinary differential equations
\begin{eqnarray}
\label{eq1.1}
\overset{\scriptscriptstyle1}\nabla_{\dot\gamma}\dot\gamma
&=&
-\frac{1}{2} g_2(\dot\tau,\dot\tau)\cdot g_1^\sharp(dk)\circ\gamma
\\
\label{eq1.2}
\overset{\scriptscriptstyle2}\nabla_{\dot\tau}\dot\tau
&=&
-\frac{1}{k\circ\gamma}dk(\dot\gamma)\cdot\dot\tau
\end{eqnarray}
where $g_1^\sharp:T^\ast M_1\to TM_1$ is the canonical isomorphism
of bundles
induced by $g_1$.
\end{lem}  
From \cite{CK}, we also get:
\begin{lem}
\label{geg}
The map $\mu:I\to M_1$ is  a geodesic with respect to the 
metric 
$G_r$ if and  only if 
\begin{equation}
\label{geG}
\overset{\scriptscriptstyle1}\nabla_{\dot\mu}\dot\mu
=\frac{1}{2k\circ\mu(1+rk\circ\mu)}\left\{2dk(\dot\mu)\cdot\dot\mu-
g_1(\dot\mu,\dot\mu)\cdot g_1^\sharp(dk)\circ\mu
\right\}\ .
\end{equation}
\end{lem}
We conclude this number by two lemmas needed in the following.
\begin{lem}
\label{th:2}
Let $\mathcal M$ be a topological space  equipped with
two distance functions
 $d_1$ and $d_2$.
Suppose that any Cauchy sequence of $d_2$ is also 
a Cauchy sequence of $d_1$.
Then the completeness of $d_1$ implies the completeness of $d_2$.
\end{lem}
A proof of the above lemma is straightforward
and we omit it here.
\par
We observe that if there exits  a positive number $L$
such that $d_1(x_1,x_2)\geq Ld_2(x_1,x_2)$, 
for each $x_1,x_2\in \mathcal M$ ,
then each Cauchy sequence of $d_2$ is also a Cauchy sequence of $d_1$.
\begin{cor}
\label{comple}
If the Inequality \eqref{boun} holds,
the manifold $(M_1,g_1)$ is complete 
if and only if there exists an $r>k_1$ such that $(M_1,G_r)$ is complete.
\end{cor}
\begin{proof}
We shall denote by $d_{g_1}$, $d_{G_r}$ the distance functions associated
with the Riemannian metrics $g_1$ and $G_r$, respectively.
\par
For any $X\in T_{x_0}M_1$ and $x_0\in M_1$, we have
\begin{eqnarray*}
g_1(X,X)=\frac{k(x_0)}{1+rk(x_0)}G_r(X,X)
\ \hbox{and}\ 
G_r(X,X)=\frac{1+rk(x_0)}{k(x_0)}g_1(X,X)\ ;
\end{eqnarray*}
for any $r>k_1$.
\par
The functions $f_1,f_2:(k_0,+\infty)\to{\mathbb R}$ defined respectively by setting
\begin{eqnarray*}
f_1(t)=\frac{t}{1+rt}
\quad\hbox{and}\quad
f_2(t)=\frac{1+rt}{t}\ ;\quad\forall r\in(k_0,+\infty)
\end{eqnarray*}
are bounded.
\par
Hence, there exist two positive real numbers $k_2$ and $k_3$ such that 
\begin{equation*}
d_{g_1}(x_1,x_2)\leq\sqrt{k_2}d_{G_h}(x_1,x_2)
\ \textrm{ and }\ 
d_{G_r}(x_1,x_2)\leq\sqrt{k_3}d_{g_1}(x_1,x_2)
\end{equation*}
for all $x_1,x_2\in M_1$.
\par
Then our corollary follows immediately from Lemma~\ref{th:2}.
\par
\end{proof}
Finally, we recall that connected, complete, Riemannian manifolds are geodesically connected (see, e. g., \cite{KL}).
\section{Geodesics on \boldmath${(M,G_r+ g_2)}$ 
and \boldmath${(M,g)}$}
In this Section we shall use the geometric objects and the notations introduced in the previous one. 
\begin{lem}
\label{phi}
For any $\mu:I\to M_1$ and for any $r>k_1$, there is a uniquely determined diffeomorphism
$\varphi_r:I\to I$ 
such that 
\begin{eqnarray}
\varphi_r(0)=0, \ \ \ 
 \varphi_r(1)=1
 \nonumber
 \\
\dot\varphi_r=a_r\frac{1+rk}{k}\circ\mu\circ\varphi_r
\label{der1}
\end{eqnarray}
where $a_r$ is a suitable real number.
\end{lem} 
\begin{proof}
We shall determine $\varphi^{-1}_r$ and then we shall obtain
$\varphi_r$ as the inverse of $\varphi^{-1}_r$.
\par
Condition \eqref{der1}
is equivalent to 
\[
\frac{d\varphi^{-1}_r}{ds}=\frac{k(\mu(s))}{a_r(1+rk(\mu(s)))}
\]
Hence the map $\varphi^{-1}_r$ is defined by
\begin{equation}
\label{defi1}
\varphi^{-1}_r(s):=\frac1{a_r}\int_0^s\frac{k}{1+rk}\circ\mu\ d\xi\quad ,\ \ \ \ \ \
a_r:=\int_0^1\frac{k}{1+rk}\circ\mu\ d\xi
\ ;
\end{equation}
for any $s\in I$.
\par
As a consequence, $\varphi^{-1}_r$ is a smooth strictly increasing
 diffeomorphism from $I$ onto $I$.
 \par
\end{proof}
We need the following lemma, too.
\begin{lem}
\label{psi1}
For any differentiable curve $\gamma:I\to M_1$, there is a uniquely determined diffeomorphism
$\psi:I\to I$, such that 
\begin{eqnarray}
\psi(0)=0,\ \ \ 
 \psi(1)=1 
 \nonumber
 \\
\dot\psi=\frac{b}{k\circ\gamma}
\ ,
\label{der2}
\end{eqnarray}
where $b$ is a suitable positive real number.
\end{lem} 
\begin{proof}
The map $\psi$ is defined by
\begin{equation}
\label{def2}
\psi(s):=b\int_0^s\frac{1}{k\circ\gamma} d\xi,
\ \ \ \
b:=\Bigl(\int_0^1\frac{1}{k\circ\gamma} d\xi\Bigr)^{-1}\ ,
\end{equation}
for any $s\in I$.
\par
\end{proof}
The previous lemma implies:
\begin{theo}
\label{thpa}
Let $\mu:I\to M_1$ and $\nu,\tau:I\to M_2$ be smooth curves and suppose $\tau=\nu\circ\psi$, being $\psi$ defined by the previous lemma.
\par
Then, $\tau$ satisfies \eqref{eq1.2}, if and only if $\nu$ is a geodesic of $g_2$. 
\par
Moreover, it results $\tau(0)=\nu(0)$ and $\tau(1)=\nu(1)$.  
\end{theo}
\begin{proof}
In fact, it results 
\begin{eqnarray*}
\overset{\scriptscriptstyle2}\nabla_{\dot\tau}\dot\tau
&=&
(\dot\psi)^2\cdot(\overset{\scriptscriptstyle2}\nabla_{\dot\nu}\dot\nu)
\circ\psi+
\Ddot\psi\cdot\dot\nu\circ\psi
=
\\
& &\overset{\eqref{der2}}=
(\dot\psi)^2\cdot(\overset{\scriptscriptstyle2}\nabla_{\dot\nu}\dot\nu)
\circ\psi+\frac{b}{k^2\circ\mu}((dk)(\dot\mu))\cdot\dot\nu\circ\psi
=\\
& &
=(\dot\psi)^2\cdot(\overset{\scriptscriptstyle2}\nabla_{\dot\nu}\dot\nu)
\circ\psi
-\frac{1}{k\circ\mu}dk(\dot\mu)\cdot\dot\tau\ ;
\end{eqnarray*}
and we have the assertion.
\par
\end{proof}

\begin{lem}
\label{red}
Let $\mu_r,\gamma_r:I\to M_1$ be two smooth curves, such that $\gamma_r=\mu_r\circ\varphi_r$, being $\varphi_r$ the mapping defined by Lemma~\ref{th:2}, with $\mu=\mu_r$. 
\par
Then, $\mu_r$ is a geodesic with respect to the metric $G_r$,
if and only if the curve $\gamma_r$ satisfies the equation:
\begin{equation}
\label{eq1.11}
\overset{\scriptscriptstyle1}\nabla_{\dot\gamma_r}\dot\gamma_r
=
\frac{-1}{2k\circ\gamma_r(1+rk_r\circ\gamma_r)}
g_1(\dot\gamma_r,\dot\gamma_r)
g_1^\sharp(dk)\circ\gamma_r\ .
\end{equation}
Moreover, we have $\mu_r(0)=\gamma_r(0)$ and $\mu_r(1)=\gamma_r(1)$.
\end{lem}
\begin{proof}
In fact, we have 
\begin{eqnarray*}
\overset{\scriptscriptstyle1}
\nabla_{\dot\gamma_r}\dot\gamma_r
&=&
(\dot\varphi_r)^2\cdot
(\overset{\scriptscriptstyle1}
\nabla_{\dot\mu_r}\dot\mu_r)\circ\varphi_r
+
\Ddot\varphi_r\cdot(\dot\mu_r\circ\varphi_r)
\\
&\overset{\eqref{geG}}=&
\frac{-\dot\varphi_r^2}{2k\circ\mu_r\circ\varphi_r(1+rk\circ\mu_r\circ\varphi_r)}
g_1(\dot\mu_r,\dot\mu_r)\circ\varphi_r\cdot g_1^\sharp(dk)
\circ\mu_r\circ\varphi_r
\\
&&
+
\frac{\dot\varphi_r^2}{k_r\circ\mu_r\circ\varphi_r(1+rk_r\circ\mu_r\circ\varphi_r)}
dk(\dot\mu_r)\circ\varphi_r\cdot\dot\mu_r\circ\varphi_r
\\
&&
+
\Ddot\varphi_r\cdot\dot\mu_r\circ\varphi_r
\\
&\overset{\eqref{defi1}}=&
\frac{-1}{2k_r\circ\gamma_r(1+rk_r\circ\gamma_r)}
g_1(\dot\gamma_r,\dot\gamma_r)
g_1^\sharp(dk)\circ\gamma_r\\
&&{}
+
\frac{1}{k_\circ\gamma_r(1+rk\circ\gamma_r)}
dk(\dot\gamma_r)\cdot\dot\gamma_r
+\Ddot\varphi_h(dk(\dot\gamma_h))\cdot
\dot\mu_h\circ\varphi_h
\\
&\overset{\eqref{der1}}=&
\frac{-1}{2k_r\circ\gamma_r(1+rk_r\circ\gamma_r)}
g_1(\dot\gamma_r,\dot\gamma_r)
g_1^\sharp(dk)\circ\gamma_r\ .
\end{eqnarray*}
Since the vice versa can be proved in an analogous way, our lemma follows.
\par
\end{proof}
\begin{lem}
\label{nor1}
Under the assumptions of the previous lemma, if either $\mu_r$ is a geodesic of $G_r$ or $\gamma_r$ verifies \ref{eq1.11},
we have 
\begin{equation}
\label{enor1}
g_1(\dot\gamma_r,\dot\gamma_r)=a_r^2\frac{(1+rk(x_0))(1+rk\circ\gamma_r)}{k(x_0)k\circ\gamma_r}
\cdot g_1(X_r,X_r)\ ,
\end{equation}
 being $\gamma_r(0)=x_0$ and $X_r=\dot\mu_r(0)$.
\end{lem}
\begin{proof}
In fact, it results
\begin{eqnarray*}
g_1(\dot\gamma_r,\dot\gamma_r)
&=&
(\dot\varphi_r)^2\cdot g_1(\dot\mu_r\circ\varphi_r,\dot\mu_r\circ\varphi_r)
\\
&\overset{\eqref{defi1}}=&
a^2_r\left(\frac{1+rk}{k}\circ\mu_r\circ\varphi_r\right)^2\cdot
g_1(\dot\mu_r\circ\varphi_r,\dot\mu_r\circ\varphi_r)\ .
\end{eqnarray*}
Then, under the assumptions of our lemma, it follows
\begin{eqnarray*}
g_1(\dot\gamma_r,\dot\gamma_r)
=
a^2_r\frac{1+rk\circ\gamma_r}{k\circ\gamma_r}\cdot G_r(\dot\mu_r\circ\varphi_r,\dot\mu_r\circ\varphi_r)\ .
\end{eqnarray*}
From which \eqref{enor1} immediately follows.
\par
\end{proof}
From the above lemma and Lemma \ref{red}, we get the following 
\begin{lem}
Under the assumptions of the previous lemma, if $\mu_r:I\to M_1$ is a geodesic with respect to the metric $G_r$
then
\begin{equation}
\label{eq:3}
\overset{\scriptscriptstyle1}
\nabla_{\dot\gamma_r}\dot\gamma_r
=
\frac{-a^2_r(1+rk(x_0))}{2k(x_0)k^2\circ\gamma_r}\cdot g_1(X_0,X_0)\cdot g_1^\sharp(dk)
\circ\gamma_r\ .
\end{equation}
\end{lem}
The next lemma characterizes the norm of the vector field
$\dot\tau_r$.
We  skip the proof of this lemma for
it  is very similar to that one of Lemma~\ref{nor1}.
\begin{lem}
Let $\mu_r:I\to M_1$ and $\tau_r,\nu:I\to M_2$ be three smooth curves such that $\tau_r=\nu\circ\psi_r$, being $\psi_r$ defined as in Lemma~\ref{psi1}, by means of $\mu_r$. 
If either $\nu_r$ is a geodesic of $g_2$ or $\tau_r$ is a solution of Equation 2.2, then
\begin{equation}
\label{nor2}
g_2(\dot\tau,\dot\tau)=
\frac{b_r^2}{k^2\circ\gamma_r}\cdot g_2(Y_0,Y_0)\ ;
\end{equation}
with $\nu(0)=y_0$ and $\dot\nu(0)=Y_0$.
\end{lem}
With the previous notations, we have:
\begin{theo}
\label{suf:2}
Suppose that
the curve $(\mu_r,\nu_r):I\to M$ is a geodesic with respect to the metric 
$G_r+ g_2$
and
\begin{equation}
\label{suf1}
a^2_r\frac{1+rk(x_0)}{k(x_0)}\cdot g_1(X_0,X_0)
=
b^2_rg_2(Y_0,Y_0)\ ;
\end{equation} 
with $\mu_r(0)=x_0$, $\nu_r(0)=y_0$, $\dot\mu_r(0)=X_0$ and $\dot\nu_r(0)=Y_0$.
\par
Then, the curve $(\gamma_r,\tau_r):I\to M$, obtained as in the previous Lemmas is a geodesic with respect
to the metric $g$.
\par
We have $(\mu_r(0),\nu_r(0))=(x_0,y_0)$ and $(\mu_r(1),\nu_r(1))=(\gamma_r(1),\tau_r(1))$, too.
\end{theo}
\begin{proof}
Since  $(\mu_r,\nu_r):I\to M$ is  a geodesic 
of the metric $G_r+ g_2$  then
 $\mu_r:I\to M_1$ is a geodesic of $G_r$ and
$\nu_r:I\to M_2$ is a geodesic of $g_2$.
Hence from Theorem~\ref{thpa} it follows that 
the curve $(\gamma_r,\tau_r)$ satisfies Equation~\eqref{eq1.2}.
\par
As a consequence, we need only to prove that $(\gamma_r,\tau_r)$ satisfies
Equation~\eqref{eq1.1}.
In fact, we have
\begin{eqnarray*}
\overset{\scriptscriptstyle1}\nabla_{\dot\gamma_r}\dot\gamma_r
&\overset{\eqref{eq:3}}=&
\frac{-a^2_r(1+rk(x_0))}{2k(x_0)k^2\circ\gamma_r}\cdot g_1(X_0,X_0)\cdot g_1^\sharp(dk)
\circ\gamma_r
\\
&
\overset{\eqref{suf1}}=&
\frac{-b^2_r}{2k^2\circ\gamma_r}g_2(Y_0,Y_0)\cdot g_1^\sharp(dk)
\circ\gamma_r
\\
&\overset{\eqref{nor2}}=&
\frac{-1}{2} g_2(\dot\tau_r,\dot\tau_r)\cdot
g_1^\sharp(dk)\circ\gamma_r\ .
\end{eqnarray*}
\end{proof}
Hence, we put the following definition.
\begin{defi}
\rm 
Let $(\mu_r,\nu_r): I\to M$ be a geodesic of $G_r+ g_2$
and let $(\gamma_r,\tau_r)$ be the geodesic of $(M,g)$ 
obtained via the reparametrization by the functions $\varphi_r$ and $\psi_r$ from $(\mu_r,\nu_r)$.
\par
Then, $(\gamma_r,\tau_r)$  is called \emph{Riemannian geodesic 
of $(M,g)$}.
\end{defi}
\begin{rem}
\label{a}
Under the assumptions of the previous theorem we set:
\begin{eqnarray}
\mu_r(0)=x_0=\gamma_r(0)\ ,\
\dot\mu_r(0)=X_0=X_r\ ,\ \dot\gamma_r(0)=\widetilde{X}_r\ 
\end{eqnarray}
and
\begin{eqnarray}
\nu_r(0)=y_0=\tau_r(0)\ ,\ \dot\nu_r(0)=Y_0=Y_r\ ,\dot\tau_r(0)=\widetilde{Y}_r\ .
\end{eqnarray}
Then we have:
\begin{eqnarray}
\label{tan}
\widetilde{X}_r=a_r\frac{1+rk(x_0)}{k(x_0)}X_r\quad\hbox{and}\quad\widetilde{Y}_r=\frac{b_r}{k(x_0)}Y_r\ .
\end{eqnarray}
With these notations, the first identity of \ref{suf1} can be written as
\begin{eqnarray*}
a^2_r\frac{1+rk(x_0)}{k(x_0)}\cdot g_1(X_r,X_r)
=
b^2_rg_2(Y_r,Y_r)\ ;
\end{eqnarray*}
and it is equivalent to
\begin{eqnarray}
\label{suf2}
g_1(\widetilde{X}_r,\widetilde{X}_r)
=
k(x_0)(1+rk(x_0))g_2(\widetilde{Y}_r,\widetilde{Y}_r)\ .
\end{eqnarray}
The previous equality implies that the geodesic $(\widehat{\nu}_r,\widehat{\tau}_r)$ of $g$, having $(x_0,y_0)$ and $(a\widetilde{X}_r,a\widetilde{Y}_r)$ as initial conditions, is a Riemannian geodesic of $g$, for any $a\in{\mathbb R}$.
\end{rem}
From Equation \eqref{suf2} we get
\begin{rem}
Let $\zeta_r=(\gamma_r,\tau_r)$ and $\zeta_s=(\gamma_s,\mu_s)$ be two Riemannian geodesics of $g$, with $r,s>k_1$, such that $\zeta_r(0)=\zeta_s(0)$.
\par
Then $\zeta_r=\zeta_s$, if and only if $r=s$. 
\end{rem}
\begin{theo} 
Suppose $k$ bounded and let $\widetilde{\zeta}=(\gamma,\tau):I\to M$ be a geodesic of $g$, such that $\dot\gamma(0)=\widetilde{X}_0$ and $\dot\tau(0)=\widetilde{Y}_0\not=0$.
\par
If
\begin{eqnarray}
\label{erre}
g_1(\widetilde{X},\widetilde{X})>k(x_0)g_2(\widetilde{Y},\widetilde{Y})\frac{K_0-k(x_0)}{K_0}\ ;
\end{eqnarray}
the curve $\widetilde{\zeta}$ is a Riemannian geodesic of $g$.
\end{theo}
\begin{proof}
We set 
\begin{eqnarray*}
r=\frac{g_1(\widetilde{X},\widetilde{X})}{k^2(x_0)g_2(\widetilde{Y},\widetilde{Y})}-\frac1{k(x_0)}\ .
\end{eqnarray*}
Then a symple calculation shows that $r>k_1$.
\par
Now we consider the curve $\tau$ and we set $\nu_r=\tau\circ\psi_r^{-1}:I\to M_2$, being $\psi_r$ defined by $\gamma$ as in Lemma \ref{psi1}.
\par
Since the curve $\tau$ verifies Equation \eqref{eq1.2}, the curve $\nu_r$ is a geodesic of $g_2$.
\par
Analogously, we set $\mu_r=\gamma\circ\varphi_r^{-1}$, with $\varphi_r$ defined by Lemma \ref{phi}, and $\mu_r$ is a geodesic of $G_r$, in the obvious way.
\par
Finally, the previous contruction implies that $(\gamma,\tau)$ is a Riemannian geodesic of $g$ obtained from the geodesic $(\mu_r,\nu_r)$ of $G_r+g_2$.
\par
\end{proof}
\begin{rem}
If $k$ is unbounded from above and one replaces \eqref{erre} by
\begin{eqnarray*}
g_1(\widetilde{X},\widetilde{X})>k(x_0)g_2(\widetilde{Y},\widetilde{Y})\ ;
\end{eqnarray*}
the previous theorem holds, again.
\end{rem}

\section{Some properties of Riemannian geodesics}

\begin{rem}

Let $\mu_r:I\to M_1$ be a geodesic of $G_r$, with $r>k_1$. 
\par
We recall that there exist a geodesic $\sigma_r:{\mathbb R}\to M_1$ of $G_r$ and $t_0\in{\mathbb R}$ 
such that $(\sigma_r([0,t_0])=\mu_r(I)$, being $G_r$ a complete Riemannian metric. 
\par
Moreover, it results $\dot\mu_r(0)=t_0\dot\sigma_r(0)$.
\par
An analogous statement holds for $g_2$.
\par 
This implies that the mappings $\varphi_r$ and $\psi_r$ defined respectively by Lemmas \ref{phi} and \ref{psi1} can be extended to diffeomorphims from $\mathbb R$ onto $\mathbb R$.
\par
\end{rem}
\begin{theo}
\label{pc}
Let $(\mu_r,\nu_r):{\mathbb R}\to M$ be a geodesic of $G_r+ g_2$, with $r>k_1$. 
\par
Then, for any $\alpha\in{\mathbb R}$, there exist two real numbers $\pm\beta\in{\mathbb R}$, such that the point $(\mu_r(0),\nu_r(0))$ and the point $(\mu_r(\alpha),\nu_r(\pm\beta))$ can be joined by Riemannian geodesics of $g$.
\end{theo}
\begin{proof}
We put $\dot\mu_r(0)=X_r$ and $\dot\nu_r(0)=Y_r$ and suppose $\Vert X_r\Vert_1=\Vert Y_r\Vert_2=1$, with the obvious meaning of the used symbols and without loss of generality.
\par
Then, for any $\alpha\in \mathbb R$ ($\beta\in{\mathbb R}$), the point $\mu_r(\alpha)$ ($\nu_r(\beta)$) is the end point of the geodesic of $G_r$ ($g_2$), determined by the vector $\alpha X_r$ ($\beta Y_r$).
\par
We shall denote by $a_{\alpha r}$ and $b_{\alpha r}$ the constants of Lemmas \ref{der1} and \ref{defi1} determined by means of the geodesic having $(x_0,\alpha X_r)$ as initial condition, respectively.
\par
Then, $X_{\alpha r}$ and $Y_{\beta r}$ verify Condition \eqref{suf1}, if and only if
\begin{eqnarray}
\label{beta}
a^2_{\alpha r}\frac{1+rk(x_0)}{k(x_0)}\alpha^2
=
b^2_{\alpha r}\beta^2\ .
\end{eqnarray}
Hence, the assertion follows by computing $\beta$ from \eqref{beta}.
\par
\end{proof}
\begin{theo}
\label{pc1}
Let $(\mu_r,\nu_r):{\mathbb R}\to M$ be a geodesic of $G_r+ g_2$, with $r>k_1$. 
\par
Then, for any $\beta\in{\mathbb R}$, there exist two real numbers $\pm\alpha\in{\mathbb R}$, such that the point $(\mu_r(0),\nu_r(0))$ and the point $(\mu_r(\pm\alpha),\nu_r(\beta))$ can be joined by Riemannian geodesics of $g$.
\end{theo}
\begin{proof}
The proof is analogous to the previous one.
\par
\end{proof}
\begin{cor}
\label{pm}
For any $x_0,x_1\in M_1$, for any $r>k_1$, for any geodesic $\mu_r:I\to M_2$ of $G_r$ joining $x_0$ and $x_1$ and any geodesic $\nu_r:{\mathbb R}\to M_2$ of $g_2$, there exists $\beta\in{\mathbb R}$ such that the points $(x_0,\nu(0))$ and $(x_1,\nu(\pm\beta))$ can be joined by a Riemannian geodesic of $g$, obtained in the obvious way from the previous two geodesics.
\par
An analogous statement holds for any $y_0,y_1\in M_2$.
 \end{cor}
\begin{defi}
Since Corollary \ref{pm} holds, we shall say that $M_1$ is \emph{positively and negatively geodesically connected with respect to $g$}.
\par
Analogously, we shall say that $M_2$ is \emph{positively and negatively geodesically connected with respect to $g$}.
\par
Finally, we shall say that $M$ is \emph{partially geodesically connected}, when the previous two definitions hold.
\end{defi}
\begin{theo}
Let us consider $x_0,x_1\in M_1$ and let us suppose that there exists a continuous map $X:(k_1,+\infty)\to T_{x_0}M_1$, such that for any $r\in(k_1,+\infty)$ the geodesic $\mu_r:I\to M_1$ of $G_r$, determined by the initial condition $(x_0,X(r))$, joins $x_0$ and $x_1$ and that $\mu_r$ is minimizing.
\par
Then, for any $y_0,y_1\in M_2$, there exists a Riemannian geodesic of $g$ joining $(x_0,y_0)$ and $(x_1,y_1)$.
\end{theo}
\begin{proof}
Under the assumptions of the theorem, we consider the function\\ $\beta:(k_1,\infty)\to{\mathbb R}$ defined by setting
\par
\begin{equation*}
\beta(r)=\frac{a_r}{b_r}\left(\frac{1+rk(x_0)}{k(x_0)}\cdot g_1(X(r),X(r))\right)^\frac12\ ;
\end{equation*}
where $a_r$ and $b_r$ are obtained respectively by \eqref{defi1} and \eqref{def2} along the geodesic $\mu_r:I\to M_1$ of $G_r$, for any $r\in(k_1,+\infty)$.
\par
Then, $\beta$ is continuous, too.
\par
Let $\gamma:I\to M_1$ be a minimizing geodesic of $g_1$ joining $x_0$ and $x_1$ and let us set $\dot\gamma(0)=X$.
\par
Since all the involved geodesics are minimizing, we have
\begin{eqnarray*}
g_1(X,X)a_r^{-1}\leq \frac{1+rk(x_0)}{k(x_0)}g_1(X(r),X(r))\leq
g_1(X,X)\int_0^1\frac{1+rk(\gamma(t))}{k(\gamma(t))}dt
\end{eqnarray*}
and
\begin{equation*}
g_1(X,X)\frac{a_r}{b_r^2}\leq\beta^2(r)\leq
g_1(X,X)\frac{a_r^2}{b_r^2}\int_0^1\frac{1+rk(\gamma(t))}{k(\gamma(t))}dt\ ;
\end{equation*}
\par 
The first of the previous inequalities and $k_0>0$ imply
\[
\lim_{r\to k_1}\beta(r)^2\geq\lim_{r\to-K_0^{-1}}\frac{k_0}{K_0^2(1+rK_0)}=+\infty\quad\hbox{and}\quad\lim_{r\to+\infty}\beta(r)^2=0\ .
\]
Hence, it results
\[
\lim_{r\to k_1}\beta(r)=+\infty\quad\hbox{and}\quad\lim_{r\to+\infty}\beta(r)=0\ .
\]
As a consequence of the well known generalization of the Weistrass $\beta$ is onto.
\par
Now, we consider two points $y_0,y_1\in M_2$.
\par
If $y_0=y_1$, the point $(x_0,y_0)$ and the point $(x_1,y_1)$ can be joined by a Riemannian geodesic of $g$ in a trivial way.
\par
Suppose $y_0\not=y_1$, then there exists a geodesic $\nu:{\mathbb R}\to M_2$ of $g_2$ and there exists $\beta_0\in(0,+\infty)$, such that $\nu(0)=y_0$, $g_2(\dot\nu(0),\dot\nu(0))=1$ and $\nu(\beta_0)=y_1$.
\par
Then, the geodesic of $g_2$ having $(y_0,\beta_0\dot\nu(0))$ joins $y_0$ and $y_1$.
\par
Finally, we can consider $r_0\in(k_1,+\infty)$ such that $\beta(r_0)=\beta_0$. With this choice the vectors $X_{r_0}$ and $Y_{r_0}=\beta(r_0)\dot\nu(0)$ verify \eqref{suf1}
and the assertion follows in a trivial way.
\par
\end{proof}
\begin{theo}
Suppose that the manifold $M_1$ is connected, has dimension higher than one, negative sectional curvature and that it is simply connected.
\par Suppose that $M_2$ is connected, too.
\par
Moreover, suppose that $k_0>0$ and that
\begin{eqnarray}
\label{ne}
(\overset{\scriptscriptstyle1}\nabla dk)(e,e))<
\frac{1+4rk}{2k(1+rk)}e(k)^2+
\frac{1}{4k(1+rk)}\Vert dk\Vert^2_1-k(1+rk)\overset{\scriptscriptstyle1}K(\sigma)\ ;
\end{eqnarray} 
for any vector $e\in T_xM_1$, such that $g_1(e,e)=1$ and for any $x\in M_1$.
\par
If $M_1$ and $M_2$ are geodesically connected with respect to the metrics $g_1$ and $g_2$, respectively, then for any $z_0,z_1 \in M$  there exists a Riemannian geodesic of $g$ joining $z_1$ and $z_2$.
\end{theo}
\begin{proof}
From the Appendix it follows that the sectional curvature of $G_r$ is negative, for any $r>k_1$.
\par
Since $M_1$ is simply connected, the exponential mapping of $G_r$,\\ $exp_x^r:T_xM_1\to M_1$, is a diffeomorphism, for any $x\in M_1$ (see, e. g. \cite{KL}). 
\par
Because of a theorem on the families of systems of ordinary differential equations continuously depending on a parameter, $exp_x^r$ is continuous with respect to $r>k_1$, too.
\par
Let us consider $x_0,x_1\in M_1$ and the map $X:(k_1,\infty)\to T_{x_0}M_1$ defined by setting $X(r)=(exp_{x_0}^r)^{-1}(x_1)$, for any $r\in(k_1,\infty)$.
\par
Then, $X$ is continuous and the assertion follows from the previous theorem.
\par
\end{proof}
\begin{rem}
Obviously, under the assumption of the previous theorem, for $r$ tending to $k_1$ the contribution of $k(\sigma)$ is zero, but the contribution of the second summand tends to $+\infty$.
\end{rem}
Suppose that $M_1={\mathbb R}$ and that $g_1=dt^2$ is the standard metric on $\mathbb R$.
\par
In this case the metric $g=dt^2-k(t)g_2$ coincides with the FLRW--metric (Friedman--Lemaitre--Robertson--Walker metric), with speed of light $c=1$, used in the Big Bang theories and we have:
\begin{theo}
If $M_2$ is complete with respect to the metric $g_2$ and $k$ is bounded from above and bounded from below away from zero, then for any $z_0,z_1 \in M={\mathbb R}\times M_2$  there exists a Riemannian geodesic of $g$ joining $z_1$ and $z_2$.
\end{theo}
\begin{proof}
In this case, the metric tensor $G_r$ on $\mathbb R$ is given by $G_r=(k^{-1}+r)dt^2$, for any $r>-K_0^{-1}$.
\par
Let be $r>-K_0^{-1}$,
then the Equation \eqref{geG} of a geodesic of $G_r$ becomes
\begin{eqnarray*}
\ddot\mu_r=
\frac1{2(k\circ\mu_r)(1+rk\circ\mu_r)}(k'\circ\mu_r)\dot\mu_r^2
\end{eqnarray*}
The previous equation admits a first integral given by
\[
\dot\mu_r=c_r\left(\frac{k\circ\mu_r}{1+rk\circ\mu_r}\right)^\frac12\ .
\]
Because of Corollary \ref{comple}, $\mathbb R$ is complete with respect to the metric $G_r$.
\par
Hence, we can determine $c_r$ as a solution of the equation
\[
c_r=(x_1-x_0)\left(\int_0^1\left(\frac{k(\mu_r(t))}{1+rk(\mu_r(t))}\right)^\frac12dt\right)^{-1}\ .
\]
As a consequence, the mapping $\mu_r:I\to{\mathbb R}$ is strictly increasing, for $x_1>x_0$ and strictly decreasing, for $x_1>x_0$, because the function $k$ is bounded from below by $k_0>0$.
\par
This implies that $exp^r_{x_0}:{\mathbb R}\to{\mathbb R}$ is a diffeomorphism.
\par
Since $\exp^r_{x_0}$ depends with continuity from $r\in(-K_0^{-1},+\infty)$, the proof follows as in the previous case.
\par
\end{proof}
\begin{rem}
The previous theorem holds again, if one replaces the metric $dt^2$ on $\mathbb R$ by the Riemannian metric $fdt^2$, being $f:{\mathbb R}\to{\mathbb R}$ a $C^\infty$--differentiable function such that $0<f(t)<c$, for any $t\in{\mathbb R}$, with $c\in{\mathbb R}$. 
\end{rem}
Now we return to the general case.
\begin{theo}
Let us consider $r\in(k_1,+\infty)$, a geodesic $\mu_r:{\mathbb R}\to M_1$ of $G_r$ and a geodesic $\nu_r:{\mathbb R}\to M_2$ of $g_2$.
\par
If $(\mu_r)_{\vert[0,+\infty)}$ has no auto intersections, there exists a map\\ $\theta_r:\mu_r([0,+\infty))\to\nu_r([0,+\infty))$, such that the points $(\mu_r(0),\nu_r(0))$ and\\ $(\mu_r(t),\theta_r(\nu_r(t)))$, can be joined by a Riemannian geodesic of $g$ obtained from the geodesic $(\mu_r,\nu_r):{\mathbb R}\to M$ of $G_r+ g_2$ in the obvious way and the mapping $\theta_r$ is onto.
\par
Moreover, if $\nu_r$ has no auto intersections, the mapping $\theta_r$ is one to one, too.
\end{theo}
\begin{proof}
Under the assumptions of the theorem, we set $\dot\mu_r(0)=X_r$, $\dot\nu_r(0)=Y_r$ and we suppose $\Vert X_r\Vert_1=\Vert Y_r\Vert_2=1$.
\par
We recall that, for any $t\in{\mathbb R}$, the geodesic $\mu'_r$ of $G_r$ determined by the initial conditions $(\mu_r(0),tX_r)$, has  $\mu'_r(I)\subseteq\mu_r({\mathbb R})$, joins $\mu_r(0)$ and $\mu_r(t)$ and the obvious quantities $a'_r$ and $b'_r$ are
\[
a'_r:=\int_0^1\frac{k}{1+rk}\circ\mu(\xi t)\ d\xi\quad\hbox{and}\quad (b'_r)^{-1}=\int_0^1\frac{1}{k\circ\mu_r(\xi t)} d\xi\ .
\]
\par
An analogous statement holds for $\nu_r$.
\par
Now, we notice that, since $\mu_r$ has no autointersections, we can consider the map $\mu_r^{-1}:\mu_r([0,+\infty))\to[0,+\infty)$.
\par
Moreover, the Condition \eqref{suf1} determines the mapping $\beta:[0,+\infty)\to{\mathbb R}$, defined by:
\begin{equation*}
\beta(t)=
\frac{a'_r}{b'_r}\frac{1+rk(x_0)}{k(x_0)}t\ ,\quad\forall t\in[0,+\infty)\ .
\end{equation*} 
\end{proof}
Then, we can set $\theta_r=\nu_r\circ\beta\circ\mu_r^{-1}:\mu_r([0,+\infty))\to\nu_r([0,+\infty))$.
\par
Let us consider $x_1\in\mu_r([0,+\infty))$, then exists $t\in[0,+\infty)$, such that $\mu_r(t)=x_1$, hence $t=\mu_r^{-1}(x_1)$. 
\par
Then, $\beta(t)$ is such that the vectors $tX_r$ and $\beta(t)Y_r$ verify \eqref{suf1}.
\par
As a consequence, the points $(\mu_r(0),\nu_r(0))$ and $(\mu_r(t),\nu_r(\beta(t)))$ can be joined by a Riemannian geodesic for $g$ obtained from the geodesic $(\mu_r,\nu_r)$ of $G_r+ g_2$, with $\nu_r(\beta(t))=\nu_r(\beta(\mu_r^{-1}(x_1)))=\theta_r(x_1)$.

\section{Appendix}
In this Appendix we prove the following lemma:
\begin{lem}
Suppose that the dimension of $M_1$ is higher than one and that $g_1$ has negative sectional curvature.
\par
Then, if $k$ verifies \eqref{ne}, $G_r$ has negative sectional curvature, for any $r\in(k_1,+\infty)$.
\end{lem}
\begin{proof}
Let $\Xi(M_1)$ be the Lie algebra of vector fields on $M_1$.
\par
Let us consider a connection $\nabla^h$ of $M_1$ and let us suppose $\nabla^h=\overset{\scriptscriptstyle1}\nabla+\Pi$.
\par 
Then, the curvature tensor field $R^h$ of $\nabla^h$ and the curvature tensor field $\overset{\scriptscriptstyle1}R$ of $\overset{\scriptscriptstyle1}\nabla$ are related by
\begin{eqnarray*}
& &
R^h(X,Y)Z=
\overset{\scriptscriptstyle1}R(X,Y)Z\\
& &
+\Pi(X,\Pi(Y,Z))-\Pi(Y,\Pi(X,Z))+\\
& &
(\overset{\scriptscriptstyle1}\nabla_X\Pi)(Y,Z)-
(\overset{\scriptscriptstyle1}\nabla_Y\Pi)(X,Z) ,\quad\forall X,Y,Z\in\Xi(M_1)\ .
\end{eqnarray*}

\par
Now we suppose that $h:M_1\to{\mathbb R}$ is a $C^\infty$--differentiable function and that $h(x)>0$, for any $x\in M_1$.
\par
We also suppose that $\nabla^h$ is the Levi--Civita connection of the metric tensor $hg_1$.
\par
Then, we have
\begin{eqnarray*}
\Pi(X,Y)&=&\frac1{2h}\left[X(h)Y+Y(h)X-g_1(X,Y)g_1^\sharp(d\log h)\right]\\
& &
\forall X,Y\in\Xi(M_1)\ .
\end{eqnarray*}
The two previous identities imply
\begin{eqnarray*}
& &
R^h(X,Y)Z=\overset{\scriptscriptstyle1}R(X,Y)Z+\\
& &
\frac1{2h}[(\overset{\scriptscriptstyle1}\nabla dh)(X,Z)Y-
(\overset{\scriptscriptstyle1}\nabla dh)(Y,Z)X-\\
& &
g_1(Y,Z)g_1^\sharp(\overset{\scriptscriptstyle1}\nabla_Xdh)+
g_1(X,Z)g_1^\sharp(\overset{\scriptscriptstyle1}\nabla_Ydh)]-\\
& &
\frac{1}{4h^2}[3Y(h)Z(h)X-3X(h)Z(h)Y-\\
& &
Y(h)g_1(X,Z)g_1^\sharp(dh)+
X(h)g_1(Y,Z)g_1^\sharp(dh)+\\
& &
g_1(Y,Z)\Vert d\log h\Vert_1^2X-
g_1(X,Z)\Vert d\log h\Vert_1^2Y]\ ,\quad\forall X,Y,Z\in\Xi(M_1)\ .
\end{eqnarray*}
Let $\sigma=<\{e_1,e_2\}>$ be a two dimensional subspace of $T_xM_1$, with $x\in M_1$ and let us suppose $\Vert e_1\Vert_1=\Vert e_2\Vert_1=1$ and $g_1(e_1,e_2)=0$.
\par
Then, the sectional curvature of $\nabla^h$ is
\begin{eqnarray*}
& &
K(\sigma)=\frac1{h}\overset{\scriptscriptstyle1}K(\sigma)-
\frac1{2h^2}[(\overset{\scriptscriptstyle1}\nabla dh)(e_1,e_1)+
(\overset{\scriptscriptstyle1}\nabla dh)(e_2,e_2)]-\\
& &
\frac{1}{4h^3}[3e_1(h)^2+
3e_2(h)^2-\Vert dh\Vert^2_1]\ ;
\end{eqnarray*}
being $\overset{\scriptscriptstyle1}K$ the sectional curvature of $\overset{\scriptscriptstyle1}\nabla$.
\par 
Now we suppose $h=k^{-1}+r$, where $k$ is the mapping used in the previous numbers and $r>K_0^{-1}=k_1$.
\par
Then, $dh=-k^{-2}dk$ and $\overset{\scriptscriptstyle1}\nabla dh=2k^{-3}dk\otimes dk-k^{-2}\overset{\scriptscriptstyle1}\nabla dk$.
\par 
Hence, the sectional curvature of $G_r$ is
\begin{eqnarray*}
& &
K_r(\sigma)=\\
& &
\frac{k}{1+rk}\overset{\scriptscriptstyle1}K(\sigma)+
\frac{1}{2(1+rk)^2}[(\overset{\scriptscriptstyle1}\nabla dk)(e_1,e_1)+
(\overset{\scriptscriptstyle1}\nabla dk)(e_2,e_2)]-\\ 
& &
\frac{1+4rk}{4k(1+rk)^3}[e_1(k)^2+e_2(k)^2]-\\
& &
\frac{1}{4k(1+rk)^3}\Vert dk\Vert^2_1\ ;
\end{eqnarray*} 
for any two dimensional subspace $\sigma\subseteq T_xM_1$, for any $(e_1,e_2)$ basis of $\sigma$ such that $\Vert e_1\Vert_1=\Vert e_2\Vert_1=1$ and $g_1(e_1,e_2)=0$ and for any $x\in M_1$.
\par
As a consequence, the sectional curvature of $G_r$ is negative, for any $r>k_1$, if and only if
\begin{eqnarray}
\label{bi}
& &
(\overset{\scriptscriptstyle1}\nabla dk)(e_1,e_1)+
(\overset{\scriptscriptstyle1}\nabla dk)(e_2,e_2)<\\ 
\nonumber
& &
\frac{1+4rk}{2k(1+rk)}[e_1(k)^2+e_2(k)^2]+\\
& &
\nonumber
\frac{1}{2k(1+rk)}\Vert dk\Vert^2_1-2k(1+rk)\overset{\scriptscriptstyle1}K(\sigma)\ ;
\end{eqnarray} 
We notice that, if the sectional curvature $\overset{\scriptscriptstyle1}K$ of $g_1$ is positive, then the Inequality \eqref{bi} can not hold for any $r>k_1$.
\par
Hence, we are forced to suppose the $g_1$ has either a negative or null sectional curvature.
\par
In this case, the Inequality \eqref{bi} holds, if and only if, it results
\begin{eqnarray}
\label{bid}
(\overset{\scriptscriptstyle1}\nabla dk)(e,e)<
\frac{1+4rk}{2k(1+rk)}e(k)^2+
\frac{1}{4k(1+rk)}\Vert dk\Vert^2_1-k(1+rk)\overset{\scriptscriptstyle1}K(\sigma)\ ;
\end{eqnarray} 
for any $e\in T_xM_1$, such that $g_1(e,e)=1$ and any $x\in M_1$.
\par
\end{proof}

\end{document}